\documentclass{amsproc}

\usepackage{amsmath,amsfonts,amssymb,amsthm,amscd,mathdots,eucal}
\usepackage{color}
\usepackage{tikz}
\usepackage[top=1.4in, bottom=1.4in, left=1.7in, right=1.7in]{geometry}

\newtheorem{thm}{Theorem}[section]
\newtheorem{prop}[thm]{Proposition}

\newtheorem{lem}[thm]{Lemma}

\theoremstyle{remark}
\newtheorem{rem}{Remark}[section]

\theoremstyle{remark}

\numberwithin{equation}{section}

\newcommand{\be}{\begin{equation}}  \newcommand{\ee}{\end{equation}}
\newcommand{\bea}{\begin{eqnarray*}}  \newcommand{\eea}{\end{eqnarray*}}

\begin{document}

\def\C{{\mathbb C}}
\def\R{{\mathbb R}}
\def\F{{\mathcal F}}
\def\T{{\mathcal T}}
\def\E{{\mathcal E}}
\def\Z{{\mathbb Z}}
\def\N{{\mathbb N}}

\def\L{{\Lambda}}
\def\M{{\mathcal M}}
\def\P{{\mathcal P}}
\def\D{{\mathcal D}}

\def\l{\lambda}
\def\sgn{\mbox{\rm sgn}}
\def\varen{\varepsilon}

\def\Re{\mbox{\rm Re\;}}

\title{Stochastic B\"acklund transformations}

\author{Neil O'Connell}
\address{Mathematics Institute, University of Warwick, Coventry CV4 7AL, UK and 
School of Mathematics, Trinity College Dublin, Dublin 2, Ireland}

\maketitle

\thispagestyle{empty}

\begin{abstract}
How does one introduce randomness into a classical dynamical system in order to produce something 
which is related to the `corresponding' quantum system?  We consider this question from 
a probabilistic point of view, in the context of some integrable Hamiltonian systems.
\end{abstract}


\section{Introduction}

Let $\mu\ge 1/2$ and consider the evolution $\dot x=\mu/x$ on the positive half-line.  
Then $\ddot x=-\mu^2/x^3$,
which is the equation of motion for the rational Calogero-Moser system with Hamiltonian
$$\frac12 p^2 - \frac{\mu^2}{2x^2}.$$
If we add noise, that is, if we consider the stochastic differential equation 
$$dX=dB+\frac{\mu}{X}dt,$$
where $B$ is a standard one-dimensional Brownian motion, then $X$
is a diffusion process on the positive half-line with infinitesimal generator 
$$L=\frac12\partial_x^2+\frac{\mu}{x}\partial_x.$$
The assumption $\mu\ge 1/2$ ensures that $X$ 
never hits zero.  The operator $L$ is related to the quantum Calogero-Moser Hamiltonian
$$H=\frac12\partial_x^2-\frac{\mu(\mu-1)}{2x^2},$$
via the ground state transform $$L=\psi(x)^{-1} H \psi(x),$$ where $\psi(x)=x^\mu$.
Ignoring for the moment the discrepancy between the coupling constants $\mu^2/2$ and 
$\mu(\mu-1)/2$, this provides a very simple example of a classical Hamiltonian system which 
has the property that if we add noise in a suitable way we obtain a diffusion process whose infinitesimal
generator is simply related to the corresponding quantum system.  

Appealing as it is, this example is quite unique and, in fact, somewhat misleading.  
The only constant of motion is the Hamiltonian itself and, as $\dot x=p$, the 
evolution $\dot x=\mu/x$ necessarily has $p^2/2-\mu^2/2x^2=0$. 
It is not clear how to extend this 
 construction---{\em together with its stochastic counterpart}---to allow for other values.  
In fact, there is a kind of `explanation' for this limitation which will come later.  

In the papers~\cite{noc,noc-toda} a certain probabilistic relation between the 
classical and quantum Toda lattice was observed.  This relation can be loosely described as follows:
starting with a particular construction of the classical flow on a given sub-Lagrangian manifold, {\em adding 
white noise to the constants of motion} yields a diffusion process whose infinitesimal generator is simply related 
to the corresponding quantum system.  

As we shall see, this relation extends naturally to some other integrable many-body systems,
specifically rational and hyperbolic Calogero-Moser systems.   The basic construction can be formulated
in terms of {\em kernel functions} and {\em B\"acklund transformations}.  
 For more background on the 
(interrelated) role of kernel functions and B\"acklund transformations in integrable systems see, for 
example,~\cite{hr1,ks,pg,sklyanin} and references therein.
In the present paper, to illustrate the main ideas, we will focus on rank-one (two particle) systems
although most of the constructions extend naturally to higher rank systems. 

The examples we consider are of course very special, having the property that there are kernel
functions which unite the classical and quantum systems through a kind of exact stationary
phase property.  Nevertheless, they should provide a useful benchmark for exploring 
similar relations for other Hamiltonian systems.

The outline of the paper is as follows.
In the next section, we will illustrate the basic construction of~\cite{noc,noc-toda} in the context of the rank one 
Toda lattice.  In this setting it is closely related to earlier results of Matsumoto and Yor~\cite{my} and Baudoin~\cite{baudoin}.
In Sections 3, 4 and 5, we give analogous constructions for the rational and hyperbolic Calogero-Moser systems.
As we shall see, the above example should in fact be seen as a particular degeneration of a more general construction
for the hyperbolic Calogero-Moser system, based on the kernel functions of Halln\"as and Ruijenaars~\cite{hr1,hr}.
In Section 6, we conclude with some remarks on how the solution to the Kardar-Parisi-Zhang equation can also be 
interpreted from this point of view.

\bigskip

\noindent {\em Notation.} The following notation will be used throughout.  If $E$ is a topological space, 
we denote by $B(E)$ the set of Borel measurable functions on $E$, by $C_b(E)$ the set of bounded 
continuous functions on $E$ and by $\P(E)$ the set of Borel probability measures on $E$.  
If $E$ is an open subset of $\R^n$, we denote by $C^2_c(E)$ the set of continuously twice
differentiable, compactly supported, functions on $E$.

\bigskip

\noindent {\em Acknowledgements.}  Thanks to Simon Ruijenaars for valuable 
discussions and comments on an earlier draft, and Mark Adler and Tom Kurtz for helpful correspondence.

\section{The Toda lattice}

For the rank-one Toda lattice we consider the kernel function
$$K(x,u)=\exp\left(-e^{-x}\cosh u\right),$$
and note that $K$ satisfies
\be\label{bac3}
(\partial_x\ln K)^2-(\partial_u\ln K)^2=e^{-2x},
\ee
and
\be\label{bac4}
\partial_x^2\ln K-\partial_u^2\ln K=0.
\ee
The corresponding B\"acklund transformation
\be\label{bt-t1}
\dot u = -\partial_u \ln K = e^{-x} \sinh u,\qquad \dot x =\partial_x\ln K = e^{-x} \cosh u
\ee
has the property that, if \eqref{bt-t1} holds, then
$x$ satisfies the equations of motion of the Toda system with Hamiltonian $$\frac12 p^2-\frac12 e^{-2x},$$
and $\dot u=\l$ is a conserved quantity for the coupled system.
Indeed, differentiating \eqref{bac3} with respect to $x$ yields
\be\label{bac1}
\ddot x = \partial_x^2\ln K\ \partial_x\ln K-\partial_u\partial_x\ln K\ \partial_u\ln K  = -e^{-2x},
\ee
and differentiating \eqref{bac3} with respect to $u$ gives
\be\label{bac2}
\ddot u = \partial_u^2\ln K\ \partial_u\ln K-\partial_x\partial_u\ln K\ \partial_x\ln K  =  0.
\ee
It also follows from \eqref{bac3} that $\l$ is an eigenvalue of the Lax matrix 
$$\begin{pmatrix} p & e^{-x} \\ -e^{-x} & -p \end{pmatrix} .$$

Now the equation $\dot u=\l$ is equivalent to the critical point equation 
$\partial_u \ln K_\l=0$, where $K_\l=e^{\l u} K$.  Using this equation,
namely
\be\label{cpt}
\sinh u = \l e^x,
\ee
we can rewrite the evolution equations \eqref{bt-t1} as
\be\label{eet}
\dot u =  \l,\qquad \dot x = \l + e^{-u-x} = (\partial_x+\partial_u)\ln K_\l .
\ee
We note that \eqref{cpt} has a unique solution $u_\l(x)=\sinh^{-1}(\l e^x)$ for any $\l,x\in\R$.
The relation \eqref{cpt} is stable under the new evolution equations \eqref{eet},
and is now required to be in force in order to guarantee that $(x,p)$ evolves according to 
the Toda flow on the iso-spectral manifold corresponding to $\l$.  Given any $\l\in\R$,
the evolution equations \eqref{eet} are well-posed on the corresponding iso-spectral manifold
(defined in these coordinates by the relation \eqref{cpt})
in the sense that they admit a unique semi-global solution.  For any $\l\in\R$ and initial
condition $x(0)=x_0$, the solution is given explicitly for all $t\ge 0$ by 
$$u(t)=u_\l(x_0)+\l t,\qquad
x(t)=\begin{cases} \ln\big(\frac1\l\sinh u(t)\big) & \l\ne 0\\ \ln\big(e^{x_0}+t\big) & \l=0.\end{cases}$$
The evolution equations \eqref{eet} provide the correct framework into which we can introduce noise with the desired outcome.  

Let $H=(\partial_x^2-e^{-2x})/2$, and write $H_\l=H-\l^2/2$.
Combining \eqref{bac3} and \eqref{bac4} gives the intertwining relation
\be\label{int1-t}
H_\l K_\l= \big(\frac12\partial_u^2-\l\partial_u\big) K_\l .
\ee 
It follows, using the Leibnitz rule, that
$$\psi_\l(x)=\int_{-\infty}^\infty K_\l(x,u) du$$
is an eigenfunction of $H$ with eigenvalue $\l^2/2$. 
We note that $\psi_\l(x)=2K_\l(e^{-x})$, where $K_\nu(z)$ is the modified Bessel function of the second kind, 
also known as Macdonald's function.  

\begin{rem}
The intertwining relation \eqref{int1-t} and associated integral formula 
for the eigenfunctions can be seen as a special case of those obtained by Gerasimov, Kharchev, Lebedev
and Oblezin \cite{gklo} for the $n$-particle open Toda chain.  The above B\"acklund transformation is a special 
case of the one given by Wojciechowski~\cite{w} which, 
as remarked in that paper, is closely related to a construction of Kac and Van Moerbeke~\cite{kvm}
for the periodic Toda chain.  It can also be seen as a particular degeneration of the B\"acklund transformation
for the infinite particle system given in Toda's monograph~\cite{toda}.
\end{rem}

Consider the integral operator defined, for suitable $f:\R^2\to\R$, by
$$\tilde K_\l f(x) = \int_{-\infty}^\infty K_\l(x,u) f(x,u) du,$$
and the differential operator, defined on $\D(A_\l)=C^2_c(\R^2)$, by
$$A_\l=\frac12\partial_x^2+\frac12\partial_u^2+\partial_x\partial_u+\l\partial_u+\left(\l+e^{-u-x}\right) \partial_x.$$
\begin{prop}\label{p-int-t} For $f\in \D(A_\l)$,
\be\label{int-t}
H_\l \tilde K_\l f= \tilde K_\l A_\l f.
\ee 
\end{prop}
\begin{proof}  This follows from the intertwining relation \eqref{int1-t}.  Recall that
$$(\partial_x+\partial_u) \ln K_\l=\l+e^{-u-x}.$$
By Leibnitz' rule and integration by parts,
\begin{align*}
H_\l \tilde K_\l f (x) = H_\l \int_{-\infty}^\infty &K_\l f du \\
= \int_{-\infty}^\infty  \Big[ (H_\l K_\l) f + & (\partial_x K_\l)\partial_x f+ K_\l \tfrac12 \partial_x^2 f\Big] du \\
=  \int_{-\infty}^\infty  \Big[ (\tfrac12 \partial_u^2 K_\l-\l\partial_u K_\l) f &+K_\l (\partial_x \ln K_\l)\partial_x f+ K_\l\tfrac12 \partial_x^2 f\Big] du \\
= \int_{-\infty}^\infty\Big[ K_\l (\tfrac12 \partial_u^2 f+\l\partial_u f)  +&K_\l ((\partial_x+\partial_u) \ln K_\l)\partial_x f
-(\partial_u K_\l)\partial_x f + K_\l\tfrac12 \partial_x^2 f\Big] du \\
= \int_{-\infty}^\infty\Big[ K_\l (\tfrac12 \partial_u^2 f+\l\partial_u f) +&K_\l ((\partial_x+\partial_u) \ln K_\l)\partial_x f
+  K_\l\partial_u\partial_x f + K_\l\tfrac12 \partial_x^2 f\Big] du \\
= K_\l A_\l f,&
\end{align*}
as required.
\end{proof}

Now, if $\l\in\R$, the intertwining relation \eqref{int-t} has a {\em probabilistic} meaning, which we will
soon make precise.  It implies that there is a two-dimensional diffusion process, characterized
by the differential operator $A_\l$, which has the property that,
with particular initial condition specified by the kernel $K_\l$, its projection onto the $x$-coordinate
 is a diffusion process in $\R$
which is characterised by a renormalisation of the operator $H_\l$.  Moreover, the two-dimensional 
diffusion process charaterized by $A_\l$ is precisely the B\"acklund transformation, in the form of \eqref{eet},
with white noise added to the constant of motion $\l$.  We will now make this statement precise.

Suppose $\l\in\R$, let $B$ be a standard one-dimensional Brownian motion and consider the coupled
stochastic differential equations obtained by adding white noise to $\l$ in \eqref{eet}, that is
\be\label{sdet}
dU=dB+\l dt,\qquad dX=dU +e^{-U-X} dt.
\ee
This can be solved explictly: for any initial condition $(X_0,U_0)$, 
\be\label{sol}
U_t= U_0+B_t+\l t,\qquad X_t=U_t+\ln\left( e^{X_0-U_0}+\int_0^t e^{-2U_s} ds \right).
\ee
As the function $(x,u)\mapsto (\l+e^{-u-x},\ \l)$ is locally Lipschitz, it follows that,
for any initial condition, \eqref{sol} is the unique solution to \eqref{sdet}.
Moreover, it is a diffusion process in $\R^2$ with infinitesimal generator $A_\l$
and the martingale problem for $(A_\l,\nu)$ is well-posed for any $\nu\in\P(\R^2)$.
For more background on the relation between stochastic differential equations
and martingale problems see, for example,~\cite{ek,k11}.

Next we consider the diffusion process on $\R$ with infinitesimal generator
$$L_\l=\psi_\l(x)^{-1} H_\l \psi_\l(x)=\frac12\partial_x^2+\partial_x\ln \psi_\l(x) \cdot \partial_x.$$
This process was introduced by Matsumoto and Yor~\cite{my}.
Observe that the drift
$$\partial_x \ln \psi_\l(x)=e^{-x} \frac{K_{\l+1}(e^{-x})}{K_\l(e^{-x})},$$
is locally Lipschitz, behaves like $e^{-x}$ at $-\infty$ and vanishes at $+\infty$.  It follows that $-\infty$
is an entrance boundary, $+\infty$ is a natural boundary and, for any $\rho\in\P(\R)$, the martingale 
problem for $(L_\l,\rho)$, with $\D(L_\l)=C^2_c(\R)$, is well-posed.

Using the theory of Markov functions (see Appendix A), the 
intertwining relation \eqref{int-t} yields the following result
of Matsumoto and Yor~\cite{my} and Baudoin~\cite{baudoin}.

\begin{thm}\label{my}
Let $\rho\in\P(\R)$ and $\nu= \rho(dx) \nu_x(du)\in\P(\R^2)$,
where $\nu_x(du)= \psi_\l(x)^{-1}K_\l(x,u)du$.
Let $(X,U)$ be a diffusion process in $\R^2$ with initial condition $\nu$ and infinitesimal generator $A_\l$.
Then $X$ is a diffusion process in $\R$ with infinitesimal generator $L_\l$.
Moreover, for each $t\ge 0$ and $g\in B(\R)$, 
$$E[g(U_t)|\ X_s,\ 0\le s\le t]=\int_{-\infty}^\infty g(u) \nu_{X_t}(du) ,$$
almost surely.
\end{thm}
\begin{proof}  This follows from the intertwining relation \eqref{int-t}, using Theorem~\ref{mf}.  
The map $\gamma:\R^2\to \R$ defined by $\gamma(x,u)=x$ 
is continuous and the Markov transition kernel $\L$ from $\R$ to $\R^2$ defined by
$$\L f(x)= \int_{-\infty}^\infty \nu_x(du) f(x,u),\qquad f\in B(\R)$$
satifies $\L (g\circ\gamma)=g$ for $g\in B(\R)$.
Moreover, by  \eqref{int-t}, 
\be\label{int-t-L}
L_\l \L f=\L A_\l f,\qquad f\in\D(A_\l).
\ee
Now, $\D(A_\l)=C^2_c(\R^2)$ is closed under multiplication, separates points and is convergence 
determining. Finally, by It\^o's lemma and the intertwining relation \eqref{int-t-L}, the martingale problem for 
$(L_\l,\rho)$, now taking $\D(L_\l)=\L (\D(A_\l))\cup C^2_c(\R)$, is also well-posed, so we are done.
\end{proof}

\begin{figure}\label{fig}
\begin{center}
\includegraphics[scale=0.6,angle=90]{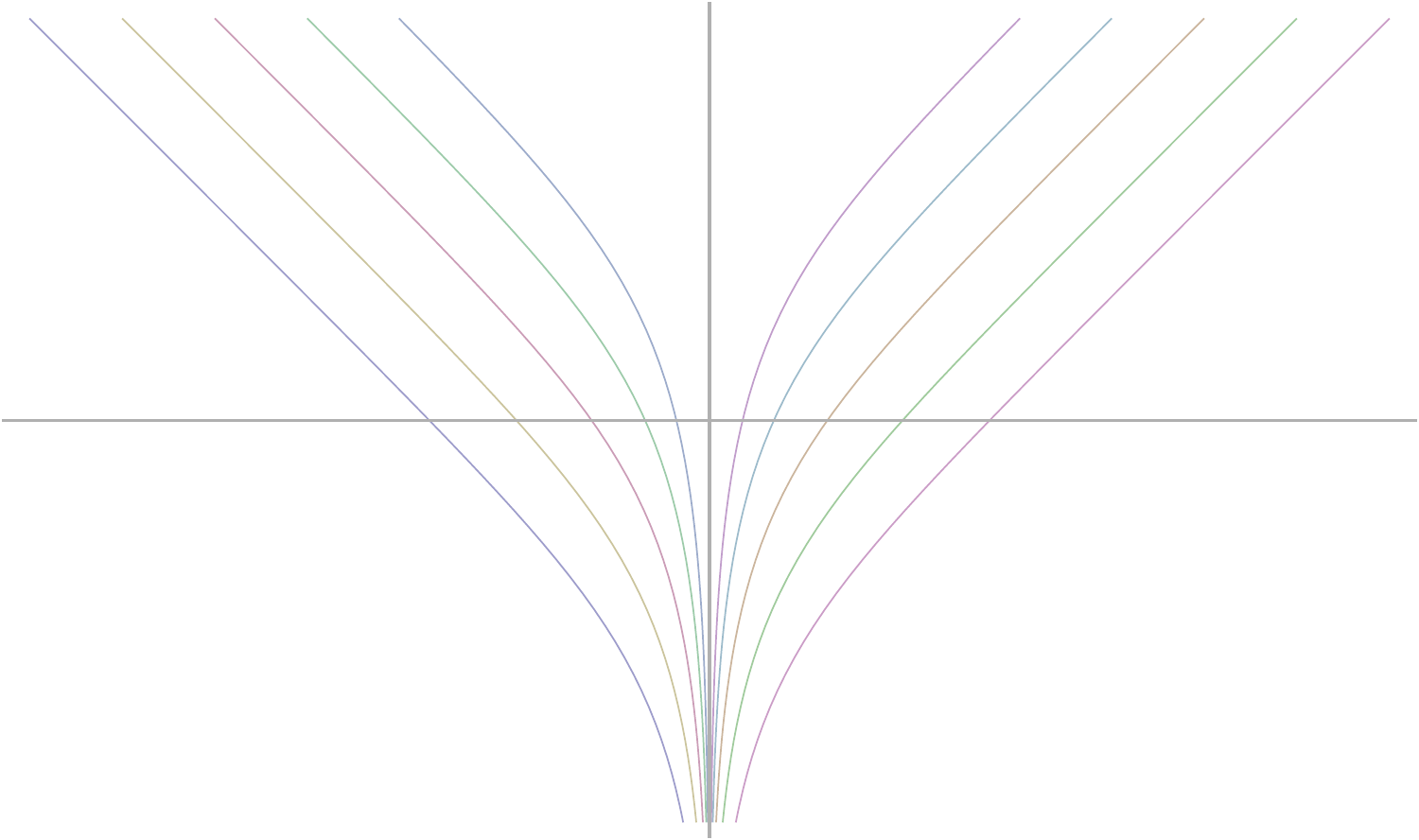}
\caption{The Toda flow in $u$ (horizontal) and $x$ (vertical) coordinates}
\end{center}
\end{figure}

To summarise, for any given value of the constant of motion $\l=\dot u\in\R$, 
the classical flow in $\R^2$ is along the curve $\sinh u=\l e^x$ (see Figure 1),
according to the evolution equations
\be\label{1}
\dot u =  \l,\qquad \dot x = \dot u + e^{-u-x},
\ee
and the $x$-coordinate satisfies the equation of motion $\ddot x=-e^{-2x}$.  
If we add noise to the constant of motion $\l$, 
then the evolution is described by the stochastic differential equations
\be\label{2}
dU=dB+\l dt,\qquad dX=dU + e^{-U-X} dt
\ee
and, for appropriate (random) initial conditions, the $u$-coordinate evolves as a Brownian
motion with drift $\l$ and the $x$-coordinate evolves as a diffusion process in $\R$ 
with infinitesimal generator $L_\l$.
As \eqref{1} is essentially a rewriting of the B\"acklund transformation~\eqref{bt-t1}, 
and in view of Theorem~\ref{my}, it seems natural to refer to \eqref{2} as a 
{\em stochastic B\"acklund transformation}, hence the title of this paper.  

To relate this to the semi-classical limit, consider the Hamiltonian
$$H^{(\epsilon)} =\frac{\epsilon}2\partial_x^2-\frac1\epsilon e^{-2x}.$$  Now the eigenfunctions are given by
$$\psi^{(\epsilon)}_\l(x)=\int_{-\infty}^\infty K_\l(x,u)^{1/\epsilon} du,$$
and Theorem~\ref{my} can be restated as follows.  Let $X_0=x$ and choose $U_0$
at random according to the probability distribution 
$$\nu^{(\epsilon)}_x=\psi^{(\epsilon)}_\l(x)^{-1} K_\l(x,u)^{1/\epsilon} du.$$
Let $(X,U)$ be the unique solution to the SDE
$$dU=\sqrt{\epsilon} dB+\l dt,\qquad dX=dU + e^{-U-X} dt,$$
with this initial condition.  Then $X$ is a diffusion process in $\R$ with 
infinitesimal generator given by
$$\frac\epsilon2\partial_x^2+\epsilon\ \partial_x\ln \psi^{(\epsilon)}_\l(x) \cdot \partial_x.$$
As $\epsilon\to 0$, the evolution of $(X,U)$ reduces to the evolution equations \eqref{eet}
and the initial distribution of $U_0$ concentrates on the unique solution $u_\l(x)$ to the
critical point equation $\partial_u\ln K_\l=0$.  On the other hand, one might expect
$$\epsilon\ \partial_x\ln \psi^{(\epsilon)}_\l(x)\to \partial_x\big[ \ln K_\l(x,u_\l(x))\big],$$
as is indeed the case, and the evolution of $X$ reduces to the gradient flow
$$\dot x = \partial_x\big[ \ln K_\l(x,u_\l(x))\big],$$
 which is equivalent to \eqref{eet} thanks to the remarkable identity
$$ \partial_x\big[ \ln K_\l(x,u_\l(x))\big]=\big[ \partial_x \ln K_\l](x,u_\l(x)).$$

\section{Rational Calogero-Moser system}\label{sec-cg}

In this section, we formulate an analogous construction for the one-dimensional 
rational Calogero-Moser system.  Consider the kernel function
$$K(x,u)= \frac{x^2-u^2}{x},\qquad |u|\le x,$$
and note that $K$ satisfies
\be\label{bac3-rcm}
(\partial_x\ln K)^2-(\partial_u\ln K)^2=1/x^2
\ee
and
\be\label{bac4-rcm}
\partial_x^2\ln K-\partial_u^2\ln K=1/x^2.
\ee
The corresponding B\"acklund transformation
\be\label{bt-rcm}
\dot u =  \frac1{x-u}-\frac1{x+u} = -\partial_u \ln K,\qquad \dot x = \frac1{x-u}+\frac1{x+u}-\frac1x=\partial_x\ln K
\ee
has the property that, if \eqref{bt-rcm} holds, then $x$ satisfies the equations of motion of the
rational Calogero-Moser system with Hamiltonian 
$$\frac12 p^2-\frac1{2x^2},$$ 
and $\dot u=\l$ is a conserved quantity for the coupled system.
Indeed, as in the Toda case, differentiating \eqref{bac3-rcm} with respect to $x$ and $u$ yields,
respectively,
\be\label{bac1-rcm}
\ddot x = \partial_x^2\ln K\ \partial_x\ln K-\partial_u\partial_x\ln K\ \partial_u\ln K  = -1/x^3,
\ee
and
\be\label{bac2-rcm}
\ddot u = \partial_u^2\ln K\ \partial_u\ln K-\partial_x\partial_u\ln K\ \partial_x\ln K  =  0.
\ee
It also follows from \eqref{bac3-rcm} that $\l$ is an eigenvalue of the Lax matrix 
$$\begin{pmatrix} p & 1/x \\ -1/x & -p \end{pmatrix} .$$

As before, $\dot u=\l$ is equivalent to the critical point equation 
$\partial_u \ln K_\l=0$, where $K_\l=e^{\l u} K$.  Using this equation,
namely
\be\label{cp}
2u=\l(x^2-u^2),
\ee
we can rewrite the evolution equations as
\be\label{ee}
\dot u =  \l,\qquad \dot x = \l + \frac{2}{x+u}-\frac1x= (\partial_x+\partial_u)\ln K_\l .
\ee
The critical point equation \eqref{cp} has a unique solution $u_\l(x)\in(-x,x)$ for any $\l\in\R$ and $x>0$.
The relation \eqref{cp} is stable under the new evolution equations \eqref{ee},
and is now required to be in force in order to guarantee that $(x,p)$ evolves according to 
the rational Calogero-Moser flow on the iso-spectral manifold corresponding to $\l$.  
Given any $\l\in\R$,
the evolution equations \eqref{ee} are well-posed on the corresponding iso-spectral manifold
(defined in these coordinates by the relation \eqref{cp})
in the sense that they admit a unique semi-global solution.  For any $\l\in\R$ and initial
condition $x(0)=x_0>0$, the solution is given explicitly for all $t\ge 0$ by 
$$u(t)=u_\l(x_0)+\l t,\qquad
x(t)=\begin{cases} \sqrt{u(t)^2+2u(t)/\l} & \l\ne 0\\ \sqrt{x_0^2+2t} & \l=0.\end{cases}$$
As in the Toda case,
the evolution equations \eqref{ee} provide the correct framework into which we can introduce noise with the desired outcome.  

Let
$$H=\frac12\partial_x^2-\frac1{x^2},$$
and write $H_\l=H-\l^2/2$.
Combining \eqref{bac3-rcm} and \eqref{bac4-rcm} gives the intertwining relation 
\be\label{int}
H_\l K_\l= \big( \frac12\partial_u^2-\l\partial_u\big) K_\l .
\ee 
It follows that
$$\psi_\l(x)=\int_{-x}^x K_\l(x,u) du$$
is an eigenfunction of $H$ with eigenvalue $\l^2/2$.  To see this, first note that 
$$\partial_x K_\l = e^{\l u}\left( 1+\frac{u^2}{x^2}\right),\qquad \partial_u K_\l = -\frac{2u}{x} e^{\l u}+\l K_\l ,$$
and
$$K_\l(x,x)=K_\l(x,-x)=0.$$
By the Leibnitz rule,
$$\partial_x \psi_\l = \int_{-x}^x \partial_x K_\l du +K_\l(x,x)+K_\l(x,-x)=\int_{-x}^x \partial_x K_\l du,$$
and so
\bea
\partial_x^2 \psi_\l &=& \int_{-x}^x \partial_x^2 K_\l du +\partial_x K_\l(x,x)+\partial_x K_\l(x,-x)\\
&=&\int_{-x}^x \partial_x^2 K_\l du+2(e^{\l x}+e^{-\l x}).
\eea
It follows, using \eqref{int}, that
\bea
H_\l \psi_\l &=& \int_{-x}^x H_\l K_\l du + (e^{\l x}+e^{-\l x})\\
&=&  \int_{-x}^x \big( \frac12\partial_u^2 - \l\partial_u\big) K_\l du + (e^{\l x}+e^{-\l x})\\
&=& \big(\frac12 \partial_u -\l) K_\l\Big\vert^{u=x}_{u=-x} + (e^{\l x}+e^{-\l x})=0,
\eea
as required.  

\begin{rem}
The above integral representation is a special case of the {\em Dixon-Anderson formula}~\cite{for}.
The corresponding B\"acklund transformation is a special case of the one introduced in~\cite{cc1}, 
see also \cite{cc,w}.
\end{rem}

We note that $\psi_0(x)=2x^2/3$, $\psi_{-\l}(x)=\psi_\l(x)$ and, for $\l>0$,
$$\psi_\l(x)= \l^{-3/2} \sqrt{2\pi x} \ I_{3/2}(\l x),$$ where $I_\nu(z)$
is the modified Bessel function of the first kind.

Let $$D=\{(x,u)\in\R^2:\ |u|< x\}.$$ 
Consider the integral operator defined, for suitable $f:D\to\R$, by
$$\tilde K_\l f(x) = \int_{-x}^x K_\l(x,u) f(x,u) du,$$
and the differential operator, defined on $\D(A_\l)=C^2_c(D)$, by
$$A_\l=\frac12\partial_x^2+\frac12\partial_u^2+\partial_x\partial_u+\l\partial_u+\left(\l+\frac{2}{x+u}-\frac1x\right) \partial_x.$$
\begin{prop}  
For $f\in \D(A_\l)$,
\be\label{int2}
H_\l \tilde K_\l f= \tilde K_\l A_\l f.
\ee 
\end{prop}
\begin{proof}
This follows from \eqref{int}, as in the proof of Proposition~\ref{p-int-t}.
\end{proof}

Now suppose $\l\in\R$.  Let $B$ be a standard one-dimensional Brownian motion and consider the coupled
stochastic differential equations obtained by adding white noise to $\l$ in \eqref{ee}, that is
\be\label{sde}
dU=dB+\l dt,\qquad dX=dU +\left(\frac{2}{X+U}-\frac1X\right) dt.
\ee

\begin{lem} For any initial condition $\nu\in\P(D)$, the stochastic differential equation \eqref{sde} has a 
unique strong solution with continuous sample paths in $D$.  It is a diffusion process in $D$ 
with infinitesimal generator $A_\l$ and the martingale problem for $(A_\l,\nu)$ is well-posed.
\end{lem}
\begin{proof}
The function $$(x,u)\mapsto \left(\l+\frac{2}{x+u}-\frac1x,\ \l\right)$$ is uniformly Lipschitz and bounded 
on $$D_\epsilon=\{(x,u)\in D:\ x+u>\epsilon,\ x-u>\epsilon\}$$ for any $\epsilon>0$, so by
standard arguments, for any fixed initial condition $(x,u)\in D$,
the SDE \eqref{sde} has a unique strong solution with continuous sample 
paths up until the first exit time $\tau$ from the domain $D$.  We are therefore 
required to show that $\tau=+\infty$ almost surely.  As $X_t-U_t$ is non-decreasing, 
this is equivalent to showing that $Y_t=X_t+U_t$ almost surely never vanishes.
We show this by a simple comparison argument.  Set
$$b(x,u)=\frac{2}{x+u}-\frac1x=\frac{x-u}{x+u}\ \frac1x,$$
and note that for $(x,u)\in D$ with $x-u\ge \delta$, where $\delta>0$,
$$b(x,u)> \frac2{x+u}-\frac2\delta.$$
Indeed, if $x\le\delta/2$ then
$$b(x,u)=\frac{x-u}{x+u}\ \frac1x\ge \frac\delta{x+u}\frac2\delta >  \frac2{x+u}-\frac2\delta;$$
on the other hand, if $x> \delta/2$, then
$$b(x,u)=\frac{2}{x+u}-\frac1x>\frac2{x+u}-\frac2\delta.$$
Now, $$dY=2dU+b(X,U)dt,$$  
and it is straightforward to see that the one-dimensional SDE
$$dR=2 dU+\left(\frac2R-\frac2\delta\right) dt$$
has a unique strong solution with continuous sample paths in $(0,\infty)$ for any $R_0=r>0$;
by the usual boundary classification $0$ is an entrance boundary for this diffusion.
Thus, if $(X_0,U_0)=(x,u)$ and we set $\delta=x-u$ and $r=x+u$, then $Y_t\ge R_t>0$ 
almost surely for all $t\ge 0$, proving the first claim. The second claim follows.
\end{proof}

Combining this with the intertwining relation \eqref{int2}, we obtain:

\begin{thm}\label{cg-pitman}
Let $\rho\in\P((0,\infty))$ and $\nu= \rho(dx) \nu_x(du)\in\P(D)$,
where $\nu_x(du)= \psi_\l(x)^{-1}K_\l(x,u)du$.
Let $(X,U)$ be a diffusion process in $D$ with initial condition $\nu$ and infinitesimal generator $A_\l$.
Then $X$ is a diffusion process in $(0,\infty)$ with infinitesimal generator
$$L_\l=\psi_\l(x)^{-1} H_\l \psi_\l(x)=\frac12\partial_x^2+\partial_x\ln \psi_\l(x) \cdot \partial_x.$$
Moreover, for each $t\ge 0$ and $g\in B(\R)$, 
$$E[g(U_t)|\ X_s,\ 0\le s\le t]=\int_{-X_t}^{X_t} g(u) \nu_{X_t}(du),$$
almost surely.
\end{thm}
\begin{proof}  This follows from the intertwining relation \eqref{int2} using 
Theorem~\ref{mf}.  First note that we can identify $D$ with $\R^2$ via
the one-to-one mapping $(x,u)\mapsto (\ln(x+u),\ln(x-u))$ and thus regard $D$, equipped with the metric 
induced from the Euclidean metric on $\R^2$, as a complete, separable, locally compact metric space.  Similarly,
we identify $(0,\infty)$ with $\R$ via the one-to-one mapping $x\mapsto \ln x$ and regard $(0,\infty)$, equipped 
with the metric induced from the Euclidean metric on $\R$, as a complete, separable metric space. 
Note that this does not alter the topologies on $D$ and $(0,\infty)$, or the definitions of $B(D)$,
$C_b(D)$, $\P(D)$, $C^2_c(D)$, $B((0,\infty))$, $C_b((0,\infty))$, $\P((0,\infty))$, $C^2_c((0,\infty))$, 
and so on: it is just a smooth change of variables.

The map $\gamma:D\to (0,\infty)$ defined by $\gamma(x,u)=x$ is continuous and the Markov transition
kernel $\L $ from $(0,\infty)$ to $D$ defined by
$$\L f(x)= \int_{-x}^x \nu_x(du) f(x,u),\qquad f\in B(D)$$
satifies $\L (g\circ\gamma)=g$ for $g\in B((0,\infty))$.
Moreover, by  \eqref{int2}, 
\be\label{int-cg-L} 
L_\l \L f=\L A_\l f,\qquad f\in\D(A_\l).
\ee
Now, $\D(A_\l)=C^2_c(D)$ is closed under multiplication, separates points and is convergence 
determining.  Thus, all that remains to be shown is that the martingale problem for $(L_\l,\rho)$, for some
$\D(L_\l)\supset \L (\D(A_\l))$, is well-posed.  

As $\psi_\l(x)=\psi_{-\l}(x)$, we can assume $\l\ge 0$. 
The drift $b_\l(x)=\partial_x\ln\psi_\l(x)$ is given by $2/x$ if $\l=0$ and, for $\l>0$,
$$b_\l(x)=\frac1{2x}+\l\ \frac{I_{3/2}'(\l x)}{I_{3/2}(\l x)}=\frac1{2x}+\l \ \frac{I_{1/2}(\l x)+I_{5/2}(\l x)}{2I_{3/2}(\l x)}.$$
This is bounded below by $1/2x$ and converges to $\l$ as $x\to+\infty$.  
In fact, $H_\l\psi_\l=0$ implies $$\partial_x^2\ln \psi(x)=2/x^2-\l^2-b_\l(x)^2,$$ hence $b_\l(x)$ is uniformly 
Lipschitz and bounded on $(a,\infty)$ for any $a>0$.  It follows that $0$ is an entrance boundary and $+\infty$
is a natural boundary for this one-dimensional diffusion process
and the martingale problem for $(L_\l,\rho)$ with $\D(L_\l)=C^2_c((0,\infty))$ is well-posed.
By It\^o's lemma and the intertwining relation \eqref{int-cg-L}, we conclude that the martingale problem for 
$(L_\l,\rho)$ with $\D(L_\l)=\L (\D(A_\l))\cup C^2_c((0,\infty))$ is also well-posed, as required.
\end{proof}

\begin{figure}\label{fig}
\begin{center}
\includegraphics[scale=0.6,angle=90]{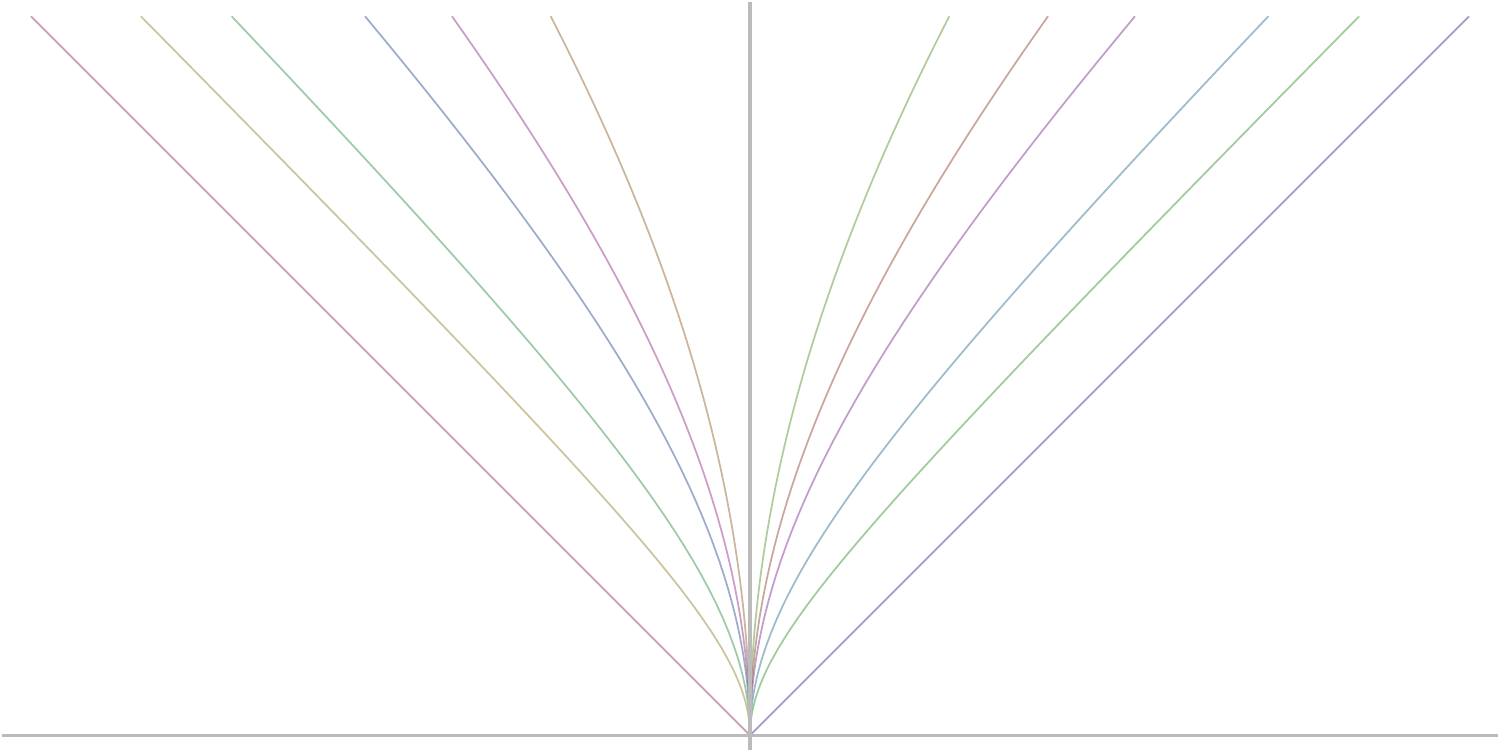}
\caption{The rational Calogero-Moser flow in $D$, shown here with $u$ as the horizontal
 and $x$ as the vertical coordinate}
\end{center}
\end{figure}

To summarise, for any given value of the constant of motion $\l=\dot u\in\R$, 
the classical flow in $D$ is along the curve $2u=\l(x^2-u^2)$ (see Figure 2),
 according to the evolution equations
$$\dot u =  \l,\qquad \dot x = \dot u + \frac{2}{x+u}-\frac1x,$$
and the $x$-coordinate satisfies the equation of motion $\ddot x=-1/x^3$.  
Adding noise to the constant of motion $\l$ gives the {\em stochastic B\"acklund transformation}
$$dU=dB+\l dt,\qquad dX=dU +\left(\frac{2}{X+U}-\frac1X\right) dt;$$
according to Theorem~\ref{cg-pitman},
for appropriate (random) initial conditions, $U$ evolves as a Brownian
motion with drift $\l$ and the $X$ evolves as a diffusion process in $(0,\infty)$ 
with infinitesimal generator $L_\l$.

When $\l=0$, as $u_0(x)=0$, the B\"acklund transformation reduces to $\dot x=1/x$,
as in the example discussed in the introduction.  Note however that in this setting
$$L_0=\frac12\partial_x^2+\frac2x \partial_x,$$
and the stochastic differential equations \eqref{sde} do {\em not} reduce to the one 
discussed in the introduction which, for example, gives the simpler construction of the
diffusion process with generator $L_0$ as the solution to the stochastic differential equation
$$dX=dB+\frac2X dt.$$

To see how the above construction relates to the semi-classical limit,
let us introduce a parameter $\mu\ge 1$ and consider
$$H=\frac1{2\mu}\partial_x^2-\frac{1+\mu}{2x^2}.$$
Then all of the above carries over with $K_\l$ replaced by $(K_\l)^\mu$ and
$$\psi^{(\mu)}_\l(x)=\int_{-x}^x K_\l(x,u)^\mu du.$$
In this setting, Theorem~\ref{cg-pitman} can be restated as follows.
Let $B$ be a Brownian motion and $(X,U)$ the unique strong solution in $D$ to
\be\label{sde-alpha}
dU=\mu^{-1/2} dB+\l dt,\qquad dX=dU +\left(\frac{2}{X+U}-\frac1X\right) dt
\ee
with $X_0=x>0$ and $U_0$ chosen at random in $(-x,x)$ according to
$$\nu_x^{(\mu)}(du) = \psi^{(\mu)}_\l(x)^{-1} K_\l(x,u)^\mu du.$$
Then $X$ evolves as a diffusion process in $(0,\infty)$ with infinitesimal generator 
$$\frac1{2\mu}\partial_x^2+\frac1\mu \partial_x\ln\psi^{(\mu)}_\l(x)\partial_x.$$
When $\mu\to\infty$, the SDE \eqref{sde-alpha} reduces to the deterministic
evolution \eqref{ee} and the initial distribution $\delta_x\times \nu_x^{(\mu)}$
concentrates on $\delta_x\times\delta_{u_\l(x)}$ where $u_\l(x)$ is the unique
solution in $(-x,x)$ to the critical point equation $\partial_u\ln K_\l=0$ 
or, equivalently $2u=\l(x^2-u^2)$.  On the other hand, one might expect
$$\frac1\mu \partial_x\ln\psi^{(\mu)}_\l(x)\to \partial_x\left[ \ln K_\l(x,u_\l(x))\right],$$
(which is indeed the case) and so in the limit as $\mu\to\infty$, the evolution of $X$ 
is according to the gradient flow
$$\dot x = \partial_x\left[ \ln K_\l(x,u_\l(x))\right].$$
Comparing this with \eqref{ee} gives, as in the Toda case,
$$ \partial_x\left[ \ln K_\l(x,u_\l(x))\right] =  \left[ \partial_x\ln K_\l\right] (x,u_\l(x)),$$
which can be verified directly.

If $-1/2 \le \mu\le 1$ and we consider 
$$H=\frac1{2}\partial_x^2-\frac{\mu(\mu+1)}{2x^2},$$
then things are more complicated, because now the evolution
$$dU=dB+\l dt,\qquad dX=dU +\mu \left(\frac{2}{X+U}-\frac1X\right) dt$$
can reach the boundary of $D$ and one needs to introduce reflecting boundary conditions
on the boundary $x+u=0$ in the $x$ direction to ensure that the appropriate intertwining 
relation holds; even then, proving the analogue of Theorem~\ref{cg-pitman} is considerably
more technical.  
One can also consider the case $-3/2\le \mu< -1/2$, but then the diffusion with infinitesimal
generator $L_\l$ will also require either reflecting (for $\mu>-3/2$) or absorbing (for $\mu=-3/2$) 
boundary conditions at zero.

Formally it can be seen that the analogue of Theorem~\ref{cg-pitman}, in the case $\mu=0$, 
corresponds to Pitman's `$2M-X$' theorem, for general drift and initial condition~\cite{pitman,rp}, 
which can be stated as follows.  Let $x\ge 0$ and $U$ be a Brownian motion 
with drift $\l$ and $U_0$ chosen at random in $[-x,x]$ with probability density proportional to $e^{\l u}$.  Set
$$X_t=U_t-\min\{2\inf_{s\le t}U_s,U_0-x\},\qquad t\ge 0.$$ 
Then $(X,U)$ is a reflected Brownian motion (with singular covariance) in the closure of $D$ 
and $X$ is a diffusion process in $[0,\infty)$ started at $x$ with infinitesimal generator 
$$\frac12\partial_x^2+\l\coth(\l x) \partial_x.$$

\section{Hyperbolic Calogero-Moser system I}

The above example extends to the hyperbolic case, taking
$$K(x,u)=\Big[\frac{\sinh\big(\epsilon\frac{x+u}2\big)\sinh\big(\epsilon\frac{x-u}2\big) }{\sinh \epsilon x}\Big]^\mu ,\qquad |u|< x.$$
We will assume for convenience that $\mu\ge 1$.  
Now,
\be\label{bac3-hcm2}
(\partial_x\ln K)^2-(\partial_u\ln K)^2=\frac{\epsilon^2\mu^2}{\sinh^2\epsilon x }
\ee
and
\be\label{bac4-hcm2}
\partial_x^2\ln K-\partial_u^2\ln K=\frac{\epsilon^2\mu}{\sinh^2\epsilon x }.
\ee
The corresponding B\"acklund transtormation
\be\label{bt-hcm2}
\dot u = -\partial_u \ln K,\qquad \dot x =\partial_x\ln K
\ee
agrees with the one given in~\cite{w} and has the property that, if \eqref{bt-hcm2} holds, 
then $x$ satisfies the equations of motion of the hyperbolic Calogero-Moser system with Hamiltonian 
\be\label{hcm-ham}
\frac12 p^2-\frac{\epsilon^2\mu^2}{2\sinh^2\epsilon x},
\ee
and $\dot u=\l$ is a conserved quantity for the coupled system.
Indeed, as before, differentiating \eqref{bac3-hcm2} with respect to $x$ and $u$ yields, respectively,
\be\label{bac1-hcm2}
\ddot x = \partial_x^2\ln K\ \partial_x\ln K-\partial_u\partial_x\ln K\ \partial_u\ln K  
=\partial_x\ \frac{\epsilon^2\mu^2}{2\sinh^2\epsilon x },
\ee
and
\be\label{bac2-hcm2}
\ddot u = \partial_u^2\ln K\ \partial_u\ln K-\partial_x\partial_u\ln K\ \partial_x\ln K  =  0.
\ee
It also follows from \eqref{bac3-hcm2} that $\l$ is an eigenvalue of the Lax matrix
\be\label{hcm-lax}
\begin{pmatrix} p & \epsilon\mu/\sinh\epsilon x  \\ - \epsilon\mu/\sinh\epsilon x  & -p \end{pmatrix} .
\ee

The equation $\dot u=\l$ is equivalent to the critical point equation $\partial_u \ln K_\l=0$.
Using this equation, namely
\be\label{cph2}
\coth\big(\epsilon\frac{x-u}2\big)-\coth\big(\epsilon\frac{x+u}2\big) = \frac{2\l}{\epsilon\mu},
\ee
we can rewrite the evolution equations \eqref{bt-hcm2} as
\be\label{eeh2}
\dot u =  \l,\qquad \dot x = \l + b(x,u) = (\partial_x+\partial_u)\ln K_\l ,
\ee
where
$$b(x,u)=(\partial_x+\partial_u)\ln K=\mu\epsilon\left[ \coth\left(\epsilon\ \frac{x+u}2\right) - \coth\epsilon x \right].$$
The equation \eqref{cph2} has a unique solution $u_\l(x)\in\R$ for each $x>0$ and $\l\in\R$.  
The relation \eqref{cph2} is stable under the new evolution equations \eqref{eeh2},
and is now required to be in force in order to guarantee that $(x,p)$ evolves according to 
the hyperbolic Calogero-Moser flow on the iso-spectral manifold corresponding to $\l$.  

Note that $u_0(x)=0$ for all $x>0$, so when $\l=0$, we must have $u(t)=0$ for all $t\ge 0$ and
the equation for $x$ simplifies to $\dot x=\epsilon\mu/\sinh\epsilon x$, which admits a unique semi-global
solution for any initial condition $x(0)=x_0>0$, defined for all $t\ge 0$ by
\be\label{h0-sol}
x(t)=\frac1\epsilon \cosh^{-1}\big(\cosh\epsilon x_0+\epsilon^2\mu t\big).
\ee

For $\l\in\R\backslash\{0\}$, the function $u_\l$ is a bijection from $(0,\infty)$ to $\R$, with inverse given by
$$u_\l^{-1}(u)=\frac2\epsilon \cosh^{-1}\sqrt{\frac{\epsilon\mu}{2\l}\sinh\epsilon u+\cosh^2\frac{\epsilon u}2}.$$
It follows that, for any given $\l\in\R$, the evolution equations \eqref{eeh2} are well-posed on the corresponding iso-spectral 
manifold (defined in these coordinates by the relation \eqref{cph2}) in the sense that they admit a unique semi-global 
solution. For $\l\in\R\backslash\{0\}$ and initial condition $x(0)=x_0>0$, 
the solution is given explicitly for all $t\ge 0$ by $u(t)=u_\l(x_0)+\l t$ and $x(t)=u_\l^{-1}(u(t))$.
As before, the equations \eqref{eeh2} provide the correct framework into which we 
can introduce noise with the desired outcome.  
 
Combining \eqref{bac3-hcm2} and \eqref{bac4-hcm2} gives the intertwining relation
\be\label{inth2}
H_\l K_\l= \big(\frac12\partial_u^2-\l\partial_u\big) K_\l
\ee 
where $K_\l=e^{\l u} K$,  $H_\l=H-\l^2/2$, and
\be\label{hcm-qham}
H=\frac12 \partial_x^2-\frac{\epsilon^2\mu(\mu+1)}{2\sinh^2\epsilon x }.
\ee
As before, it follows, using \eqref{inth2} and the Leibnitz rule, that
$$\psi_\l(x)=\int_{-x}^x K_\l(x,u) du$$
is an eigenfunction of $H$ with eigenvalue $\l^2/2$. 
Indeed, if $\mu> 1$, then $K_\l$, $\partial_x K_\l$ and $\partial_u K_\l$ vanish
for $u=\pm x$ and the claim is immediate.  If $\mu=1$, then 
$$\partial_x K_\l(x,x) = \frac\epsilon2 e^{\l x},\quad \partial_x K_\l(x,-x) = \frac\epsilon2 e^{-\l x},$$
$$ \partial_u K_\l(x,x) = -\frac\epsilon2 e^{\l x},\qquad \partial_u K_\l(x,-x) = \frac\epsilon2 e^{-\l x}$$
and
$$K_\l(x,x)=K_\l(x,-x)=0.$$
By the Leibnitz rule,
$$\partial_x \psi_\l = \int_{-x}^x \partial_x K_\l du +K_\l(x,x)+K_\l(x,-x)=\int_{-x}^x \partial_x K_\l du,$$
and so
\bea
\partial_x^2 \psi_\l &=& \int_{-x}^x \partial_x^2 K_\l du +\partial_x K_\l(x,x)+\partial_x K_\l(x,-x)\\
&=&\int_{-x}^x \partial_x^2 K_\l du+\frac\epsilon2(e^{\l x}+e^{-\l x}).
\eea
It follows, using \eqref{int}, that
\bea
H_\l \psi_\l &=& \int_{-x}^x H_\l K_\l du + \frac\epsilon4(e^{\l x}+e^{-\l x})\\
&=&  \int_{-x}^x \big( \frac12\partial_u^2 - \l\partial_u\big) K_\l du + \frac\epsilon4(e^{\l x}+e^{-\l x})\\
&=& \big(\frac12 \partial_u -\l) K_\l\Big\vert^{u=x}_{u=-x} + \frac\epsilon4(e^{\l x}+e^{-\l x})=0,
\eea
as required.  

For example, when $\mu=1$,
$$\psi_\l(x)=\frac\epsilon{\epsilon^2-\l^2}\left[ \frac\epsilon\l\coth\epsilon x\sinh\l x-\cosh\l x\right].$$
In particular, $\psi_0(x)=x\coth\epsilon x -1/\epsilon$.

Continuing as before, this kernel function leads to a hyperbolic version of Theorem~\ref{cg-pitman}, 
 valid for any $\l\in\R$.

\section{Hyperbolic Calogero-Moser system II}

There is another choice of kernel function which leads to a very
different `version' of Theorem~\ref{cg-pitman}, valid only for a restricted range of $\l$.  
It is based on the kernel
functions considered in~\cite{hr1,hr} and, in the rational case, reduces to the example discussed
in the introduction.

Let $D=(0,\infty)\times\R$ and consider the kernel function 
$$ K(x,u) = \left[ \tanh\left(\epsilon\ \frac{x+u}2\right)+\tanh\left(\epsilon\ \frac{x-u}2\right)\right]^\mu, \qquad (x,u)\in D .$$
Note that we can also write
$$K(x,u)= \left[ \frac{\sinh\epsilon x }{\cosh(\epsilon(x+u)/2)\cosh(\epsilon(x-u)/2)}\right]^\mu.$$
Now,
\be\label{bac3-hcm1}
(\partial_x\ln K)^2-(\partial_u\ln K)^2=\frac{\epsilon^2\mu^2}{\sinh^2\epsilon x }
\ee
and
\be\label{bac4-hcm1}
\partial_x^2\ln K-\partial_u^2\ln K=-\frac{\epsilon^2\mu}{\sinh^2\epsilon x }.
\ee
The corresponding B\"acklund transformation
\be\label{bt-hcm1}
\dot u = -\partial_u \ln K,\qquad \dot x =\partial_x\ln K
\ee
has the property that,  if \eqref{bt-rcm} holds, then $x$ satisfies the equations of motion of the
hyperbolic Calogero-Moser system with Hamiltonian \eqref{hcm-ham} and
$\dot u=\l$ is a conserved quantity for the coupled system, as can be seen by
differentiating \eqref{bac3-hcm1} with respect to $x$ and $u$, respectively.
It also follows from \eqref{bac3-hcm1} that $\l$ is an eigenvalue of the Lax matrix \eqref{hcm-lax}.

Now the equation $\dot u=\l$ is equivalent to the critical point equation 
$\partial_u \ln K_\l=0$, where $K_\l=e^{\l u} K$.  Using this equation,
namely
\be\label{cph}
\tanh\left(\epsilon\ \frac{x+u}2\right)-\tanh\left(\epsilon\ \frac{x-u}2\right)=\frac{2\l}{\epsilon\mu} ,
\ee
we can rewrite the evolution equations \eqref{bt-hcm2} as
\be\label{eeh}
\dot u =  \l,\qquad \dot x = \l + b(x,u) = (\partial_x+\partial_u)\ln K_\l ,
\ee
where
$$b(x,u)=(\partial_x+\partial_u)\ln K=\mu\epsilon\left[ \coth\epsilon x -\tanh\left(\epsilon\ \frac{x+u}2\right)\right].$$
In this setting, the critical point equation \eqref{cph} only has a solution $u_\l(x)\in\R$ if $|\l|<\mu\epsilon$,
in which case it is unique.  We note that $u_0(x)=0$ for all $x>0$ and $u_\l(x)\to\pm\infty$ when $\l\to\pm\mu\epsilon$.
The relation \eqref{cph} is stable under the new evolution equations \eqref{eeh},
and is now required to be in force in order to guarantee that $(x,p)$ evolves according to 
the hyperbolic Calogero-Moser flow on the iso-spectral manifold corresponding to $\l$.  

When $\l=0$, we must have $u(t)=0$ for all $t\ge 0$ and the equation for $x$ simplifies to 
$\dot x=\epsilon\mu/\sinh\epsilon x$, as in the previous example, which admits a unique 
solution for any initial condition $x(0)=x_0>0$, defined for all $t\ge 0$ by \eqref{h0-sol}.

For $\l>0$, the function $u_\l$ is a bijection from $(0,\infty)$ to $(0,\infty)$, with inverse
$$u_\l^{-1}(u)=\frac2\epsilon \cosh^{-1}\sqrt{\frac{\epsilon\mu}{2\l}\sinh\epsilon u-\sinh^2\frac{\epsilon u}2}.$$
Note that the constraint $\l<\epsilon\mu$ ensures that the quantity in the square root is positive.
For $\l<0$, $u_\l$ is a bijection from $(0,\infty)$ to $(-\infty,0)$, with inverse given by the same formula.
It follows that, given any $\l\in\R$, the evolution equations \eqref{eeh} are well-posed on the corresponding iso-spectral 
manifold (defined in these coordinates by the relation \eqref{cph}) in the sense that they admit a unique semi-global 
solution. For $\l\in\R\backslash\{0\}$ and initial condition $x(0)=x_0>0$, the solution is given for all $t\ge 0$ by 
$u(t)=u_\l(x_0)+\l t$ and $x(t)=u_\l^{-1}(u(t))$.
As before, the evolution equations \eqref{eeh} provide the correct framework into which we 
can introduce noise with the desired outcome.  

Now let 
$$H=\frac12 \partial_x^2-\frac{\epsilon^2\mu(\mu-1)}{2\sinh^2\epsilon x },$$ and write $H_\l=H-\l^2/2$.
Note that this Hamiltonian has a different coupling constant to the one in \eqref{hcm-qham}, reflecting the
difference between \eqref{bac4-hcm1} and \eqref{bac4-hcm2}.
Combining \eqref{bac3-hcm1} and \eqref{bac4-hcm1} gives the intertwining relation 
\be\label{inth}
H_\l K_\l= \big(\frac12\partial_u^2-\l\partial_u\big) K_\l
\ee 
and it follows, using the Leibnitz rule, that for $|\Re\l|<\epsilon\mu$,
$$\psi_\l(x)=\int_{-\infty}^\infty K_\l(x,u) du$$
is an eigenfunction of $H$ with eigenvalue $\l^2/2$.  

As noted in \cite[Equation (4.16)]{hr}, the eigenfunction $\psi_\l$ is related to the
associated Legendre function of the first kind by
\be\label{h-form}
\psi_\l(x)=\frac{2^{2\mu+3/2}}{\sqrt{\pi}\epsilon} (\sinh\epsilon x)^{1/2} \frac{\Gamma(\mu+\l/\epsilon)\Gamma(\mu-\l/\epsilon)}{\Gamma(\mu)}
P^{\tfrac12-\mu}_{\tfrac\l\epsilon-\tfrac12}(\cosh\epsilon x).
\ee
We note also that $\psi_\l(x)=\psi_{-\l}(x)$, as can be seen, for example, from the functional equation
$P^a_{-b}(z)=P^a_{b-1}(z)$,
and $$\psi_0(x)=\frac{2\sqrt{\pi}\Gamma(\mu)}{\epsilon\Gamma(\mu+1/2)} (\sinh\epsilon x)^\mu.$$

These are {\em not} the same eigenfunctions which were obtained in the previous section, even
taking account of the different coupling constants.  For example, taking $\mu=2$ here gives
$\psi_0(x)=8(\sinh\epsilon x)^2/3\epsilon$,  which is different from the eigenfunction 
$\psi_0(x)=x\coth\epsilon x -1/\epsilon$ of the previous section with $\mu=1$;  both are
positive on $(0,\infty)$, vanish at zero, and satisfy
$$\psi_0''-\frac{\epsilon^2}{\sinh^2\epsilon x }\psi=0,$$ 
but they are not equal.  On the other hand, they agree (up to a constant factor) in the limit 
as $\epsilon\to 0$, which corresponds to the rational case.

Consider the integral operator defined, for suitable $f:D\to\R$, by
$$\tilde K_\l f(x) = \int_{-\infty}^\infty K_\l(x,u) f(x,u) du,$$
and the differential operator, defined on $\D(A_\l)=C^2_c(D)$, by
$$A_\l=\frac12\partial_x^2+\frac12\partial_u^2+\partial_x\partial_u+\l\partial_u+\left(\l+b(x,u)\right) \partial_x.$$
\begin{prop}  
For $|\Re\l|<\epsilon\mu$ and $f\in \D(A_\l)$,
\be\label{int-h}
H_\l \tilde K_\l f= \tilde K_\l A_\l f.
\ee 
\end{prop}
\begin{proof}
This follows from \eqref{inth}, as in the proof of Proposition~\ref{p-int-t}.
\end{proof}

Now, let $B$ be a standard one-dimensional Brownian motion and consider the coupled
stochastic differential equations obtained by adding white noise to $\l$ in \eqref{eet}, that is
\be\label{sdeh}
dU=dB+\l dt,\qquad dX=dU +b(X,U) dt.
\ee

\begin{lem} Suppose $\l\in\R$ with $|\l|<\epsilon\mu$ and $\mu\ge 1/2$.
For any initial condition $\nu\in\P(D)$, the stochastic differential equation \eqref{sdeh} has a 
unique strong solution with continuous sample paths in $D$.  It is a diffusion process in $D$ 
with infinitesimal generator $A_\l$ and the martingale problem for $(A_\l,\nu)$ is well-posed.
\end{lem}
\begin{proof}
The function $(x,u)\mapsto \left(\l+b(x,u),\ \l\right)$ is uniformly Lipschitz and bounded 
on $D_\delta=\{(x,u)\in D:\ x>\delta\}$ for any $\delta>0$, so by
standard arguments, for any fixed initial condition $(x,u)\in D$,
the SDE \eqref{sdeh} has a unique strong solution with continuous sample 
paths up until the first exit time $\tau$ from the domain $D$.  We are therefore 
required to show that $\tau=+\infty$ almost surely or equivalently, that $X_t$
almost surely never vanishes. We show this by a comparison argument, using
the fact that on $D$ we have
$$b(x,u)>\mu\epsilon(\coth\epsilon x -1).$$
Now, $$dX=dU+b(X,U)dt,$$  
and it is straightforward to see that the one-dimensional SDE
$$dR= dU+\mu\epsilon(\coth(\epsilon R)-1) dt$$
has a unique strong solution with continuous sample paths in $(0,\infty)$ for any $R_0=r>0$;
since $\mu\ge 1/2$, by the usual boundary classification $0$ is an entrance boundary for this diffusion.
Thus, if $(X_0,U_0)=(x,u)\in D$ and $R_0=x-u$, then $X_t\ge R_t>0$ 
almost surely for all $t\ge 0$, as required, proving the first claim.
The second claim follows.
\end{proof}

Combining this with the intertwining relation \eqref{int-h}, we obtain:

\begin{thm}\label{cgh-pitman}
Suppose $\l\in\R$ with $|\l|<\epsilon\mu$ and $\mu> 1/2$.
Let $\rho\in\P((0,\infty))$ and $\nu= \rho(dx) \nu_x(du)\in\P(D)$,
where $\nu_x(du)= \psi_\l(x)^{-1}K_\l(x,u)du$.
Let $(X,U)$ be a diffusion process in $D$ with initial condition $\nu$ and infinitesimal generator $A_\l$.
Then $X$ is a diffusion process in $(0,\infty)$ with infinitesimal generator
$$L_\l=\psi_\l(x)^{-1} H_\l \psi_\l(x)=\frac12\partial_x^2+\partial_x\ln \psi_\l(x) \cdot \partial_x.$$
Moreover, for each $t\ge 0$ and $g\in B(\R)$, 
$$E[g(U_t)|\ X_s,\ 0\le s\le t]=\int_{-\infty}^{\infty} g(u) \nu_{X_t}(du),$$
almost surely.
\end{thm}
\begin{proof}  This follows from the intertwining relation \eqref{inth} using Theorem~\ref{mf}.  
As before, we identify $D$ with $\R^2$ via the one-to-one mapping $(x,u)\mapsto (\ln x,u)$ and 
thus regard $D$, equipped with the metric induced from the Euclidean metric on $\R^2$, as a complete, 
separable, locally compact metric space.  Similarly, we identify $(0,\infty)$ with $\R$ via the one-to-one 
mapping $x\mapsto \ln x$ and regard $(0,\infty)$, equipped with the metric induced from the Euclidean 
metric on $\R$, as a complete, separable metric space. 

The map $\gamma:D\to (0,\infty)$ defined by $\gamma(x,u)=x$ is continuous and the Markov transition
kernel $\L $ from $(0,\infty)$ to $D$ defined by
$$\L f(x)= \int_{-\infty}^\infty \nu_x(du) f(x,u),\qquad f\in B(D)$$
satisfies $\L (g\circ\gamma)=g$ for $g\in B((0,\infty))$.
Moreover, by  \eqref{int-h}, 
\be\label{int-cgh-L}
L_\l \L f=\L A_\l f,\qquad f\in\D(A_\l).
\ee
Now, $\D(A_\l)=C^2_c(D)$ is closed under multiplication, separates points and is convergence 
determining.  Thus, all that remains to be shown is that the martingale problem for $(L_\l,\rho)$, for some
$\D(L_\l)\supset \L (\D(A_\l))$, is well-posed.  

By \eqref{h-form} and the relation
$$(z^2-1) \frac{d}{dz} P^a_b(z)=bzP^a_b(z)-(a+b)P^a_{b-1}(z),$$
the drift $b_\l(x)=\partial_x\ln\psi_\l(x)$ is given by
$$b_\l(x)=\l \coth\epsilon x +\frac{\epsilon\mu-\l}{\sinh\epsilon x } 
\left[ P^{\tfrac12-\mu}_{\tfrac\l\epsilon-\tfrac32}(\cosh\epsilon x )\Big/
P^{\tfrac12-\mu}_{\tfrac\l\epsilon-\tfrac12}(\cosh\epsilon x )\right] .$$
As $x\to 0^+$, 
$$P^{\tfrac12-\mu}_{\tfrac\l\epsilon-\tfrac32}(\cosh\epsilon x ) \Big/
P^{\tfrac12-\mu}_{\tfrac\l\epsilon-\tfrac12}(\cosh\epsilon x ) \to 1.$$
Now $\mu>1/2$, so this implies that $b_\l(x)>1/2 x$ for $x$ sufficiently small,
which classifies $0$ as an entrance boundary.  On the other hand, as $x\to+\infty$,
the second term vanishes and $b_\l(x)\to\l$, which shows that $+\infty$ is a
natural boundary.  The relevant asymptotics can be found, for example, 
in \cite[\S 14.8.7, \S 14.8(iii)]{dlmf}.  Thus, as $b_\l$ is locally Lipschitz, 
 the martingale problem for $(L_\l,\rho)$ with $\D(L_\l)=C^2_c((0,\infty))$ is well-posed.
By It\^o's lemma and the intertwining relation \eqref{int-cg-L}, it follows that the martingale problem for 
$(L_\l,\rho)$ with $\D(L_\l)=\L (\D(A_\l))\cup C^2_c((0,\infty))$ is also well-posed, as required.
\end{proof}

To summarise, for any given value of the constant of motion $\l=\dot u\in\R$ with $|\l|\le\mu\epsilon$,
the classical flow in $D$ evolves according to the evolution equations
$$\dot u =  \l,\qquad \dot x = \dot u + b(x,u).$$
If we add noise to the constant of motion $\l$, 
then the evolution is described by the stochastic B\"acklund transformation
$$dU=dB+\l dt,\qquad dX=dU + b(X,U) dt$$
and, for appropriate (random) initial conditions, $U$ evolves as a Brownian
motion with drift $\l$ and $X$ evolves as a diffusion process in $(0,\infty)$ 
with infinitesimal generator $L_\l$.

As in the previous examples, we can let $\mu\to\infty$ to study the semi-classical 
limit and the result is analogous.  As before, if $\mu=1$ and $|\l|<\epsilon$, and
$u_\l(x)$ denotes the unique solution to the critical point equation $\partial_u\ln K_\l=0$, then
$$ \partial_x\left[ \ln K_\l(x,u_\l(x))\right] =  \left[ \partial_x\ln K_\l\right] (x,u_\l(x)).$$

It is natural to ask what happens to the statement of Theorem~\ref{cgh-pitman}
when $\l\to\mu\epsilon$.  In this limit, $b_\l(x)\to\mu\epsilon\coth\epsilon x$ and
$$\Gamma(\mu-\lambda/\epsilon)^{-1}\psi_\l(x)\to
\frac{2^{\mu+2}\Gamma(2\mu)}{\sqrt{\pi}\epsilon\Gamma(\mu)\Gamma(\mu+1/2)}(\sinh\epsilon x)^\mu
=: \tilde\psi_{\mu\epsilon}(x).$$
Furthermore, since $u_\l(x)\to+\infty$, it is easy to see that the measure $\nu_x$ concentrates at $+\infty$.
Now, when $\l\to\mu\epsilon$ and $u\to+\infty$, $\l+b(x,u)\to \mu\epsilon\coth\epsilon x$.
The B\"acklund transformation simplifies: if
$$\dot x=\mu\epsilon\coth\epsilon x$$
then $x$ evolves according to the hyperbolic Calogero-Moser flow with the constant to
motion $\l=\mu\epsilon$.  The statement of Theorem~\ref{cgh-pitman} carries over trivially: 
if $X$ evolves according to the SDE
$$dX=dB+\mu\epsilon\coth(\epsilon X) dt$$
then $X$ is a diffusion process on $(0,\infty)$ with infinitesimal generator
$$L_{\mu\epsilon}=\tilde\psi_{\mu\epsilon} (x)^{-1} H_{\mu\epsilon} \tilde\psi_{\mu\epsilon}(x)=
\frac12\partial_x^2+\mu\epsilon\coth\epsilon x \cdot \partial_x.$$
Similar remarks apply when $\l\to-\mu\epsilon$, and in fact the limiting statements are the same.
When $\epsilon\to 0$, this reduces to the example discussed in 
the introduction, and now we can see from the fundamental restriction $|\l|<\epsilon\mu$ that in fact 
we can only hope for the above structure to remain intact in this limit when $\l=0$,
as indeed it does with the evolution of the $x$-coordinate in \eqref{eeh} becoming
autonomous and reducing to $\dot x=\mu/x$, and the analogue of 
Theorem~\ref{cgh-pitman} carrying over trivially.

\section{The KPZ equation and semi-infinite Toda chain}

As remarked in the introduction, most of the above constructions
extend naturally to higher rank systems.  For the $n$-particle Toda chain, this has been 
developed in the papers~\cite{noc,noc-toda}.  The construction given in \cite{noc} is 
related to the geometric RSK correspondence.  In~\cite{ow} it was extended to a 
semi-infinite setting and related to the Kardar-Parisi-Zhang (or stochastic heat) equation.  
In this context it can be represented formally as a semi-infinite system of coupled stochastic 
partial differential equations, the first of which is the stochastic heat equation.  In the language 
of the present paper, the construction given in~\cite{ow} is a stochastic B\"acklund transformation 
and should be related (in a way that has yet to be fully understood) to a semi-infinite version of 
the quantum Toda chain.  See also~\cite{ch,mqr} for further related work in this direction.

With this picture in mind, it is natural to expect the construction given in \cite{ow}, without 
noise, to be related to the semi-infinite classical Toda chain.  This is indeed the case, as we will now 
explain directly.  The conclusion is that the fixed-time solution to the KPZ equation, with `narrow
wedge' initial condition, can be viewed as the trajectory of the first particle in a stochastic 
perturbation of a particular solution to the semi-infinite Toda chain.

The stochastic heat equation can be written
formally as $$u_t = \frac12 u_{xx} + \xi u$$ where $\xi(t,x)$ is space-time white noise.  It is related
to the KPZ equation $$h_t= \frac12 h_{xx} + \frac12 (h_x)^2+\xi $$ via the Cole-Hopf transformation
$h=\log u$.  The extension given in~\cite{ow} starts with a solution $u(t,x,y)$ to the stochastic
heat equation with delta initial condition $u(0,x,y)=\delta(x-y)$ and defines a sequence of
`$\tau$-functions' $\tau_n$ which can be expressed formally as the bi-Wronskians
$$\tau_n=\det[\partial_x^{i-1}\partial_y^{j-1} u]_{i,j=1,\ldots,n}.$$
Their evolution can be described, again formally, by the coupled equations
$$\partial_t a_n = \frac12 \partial_x^2 a_n + \partial_x[a_n \partial_x h_n]$$
where $a_n=\tau_{n-1}\tau_{n+1}/\tau_n^2$ and $h_n=\log(\tau_n/\tau_{n-1})$
with the convention $\tau_0=1$.  Moreover, formally it can be seen that the $\tau_n$
are $\tau$-functions for the 2d Toda chain, that is, $(\ln\tau_n)_{xy}=a_n$.

If we switch off the noise by setting $\xi=0$, then $u$ is given by the heat kernel
$$u(t,x,y)=\frac1{\sqrt{2\pi t}} e^{-(x-y)^2/2t}$$
and
$$\tau_n=t^{-n(n-1)/2} \left(\prod_{j=1}^{n-1} j!\right) u^n.$$
Note that $a_n=\tau_{n-1}\tau_{n+1}/\tau_n^2=n/t$ and 
$$h_{n+1}=-(x-y)^2/2t - \ln\big[ \sqrt{2\pi t}\ \frac{t^n}{ n!}\big].$$
These $\tau_n$ satisfy the 2d Toda equations $(\ln\tau_n)_{xy}=a_n$ as before,
but now it also holds that $(\ln\tau_n)_{xx}=-a_n$ or, equivalently,
$$(h_n)_{xx} = e^{h_n-h_{n-1}}-e^{h_{n+1}-h_n}$$
(with $h_0\equiv+\infty$) which are the equations of motion of the
semi-infinite Toda chain.

\appendix

\section{Markov functions}

The theory of Markov functions is concerned with the question: when does a function of a Markov process 
inherit the Markov property?  The simplest case is when there is symmetry in the problem, for example, the
norm of Brownian motion in $\R^n$ has the Markov property, for any initial condition, because the heat kernel 
in $\R^n$ is invariant under rotations.  A more general formulation of this idea is the well-known {\em Dynkin 
criterion}~\cite{dynkin}.  There is another, more subtle, criterion which has been proved at various levels 
of generality by, for example, Kemeny and Snell~\cite{ks}, Rogers and Pitman~\cite{rp} and Kurtz~\cite{kurtz}.  
It can be interpreted as a time-reversal of Dynkin's criterion~\cite{kelly} and provides sufficient conditions
for a function of a Markov process to have the Markov property, but only for very particular initial conditions.
For our purposes, the martingale problem formulation of Kurtz~\cite{kurtz} is best suited, as it is quite flexible 
and formulated in terms of infinitesimal generators.

Let $E$ be a complete, separable metric space. Let $A:\D(A)\subset B(E)\to B(E)$ and $\nu\in \P(E)$.
A progressively measurable $E$-valued process $X=(X_t,\ t\ge 0)$ is a solution to the {\em martingale
problem} for $(A,\nu)$ if $X_0$ is distributed according to $\nu$ and there exists a filtration $\F_t$ such that
$$f(X_t)-\int_0^t Af(X_s) ds$$
is a $\F_t$-martingale, for all $f\in\D(A)$.  The martingale problem for $(A,\nu)$ is {\em well-posed} is there exists a
solution $X$ which is unique in the sense that any two solutions have the same finite-dimensional distributions.

The following is a special case of Corollary 3.5 (see also Theorems 2.6, 2.9 and the remark at the top 
of page 5) in the paper \cite{kurtz}.

\begin{thm}[Kurtz, 1998]\label{mf} Assume that $E$ is locally compact, that $A:\D(A)\subset C_b(E) \to C_b(E)$, 
and that $\D(A)$ is closed under multiplication, separates points and is convergence determining.
Let $F$ be another complete, separable metric space, $\gamma:E\to F$ continuous and
$\L (y,dx)$ a Markov transition kernel from $F$ to $E$ such that $\L (g\circ\gamma)=g$ for all $g\in B(F)$,
where $\L f(x)=\int_E f(x) \L (y,dx)$ for $f\in B(E)$.
Let $B:\D(B)\subset B(F)\to B(F)$, where $\L (\D(A))\subset\D(B)$, and suppose
$$B\L f=\L Af,\qquad f\in\D(A).$$
Let $\mu\in\P(F)$ and set $\nu=\int_F \mu(dy) \L (y,dx)\in\P(E)$.  Suppose that the martingale problems
for $(A,\nu)$ and $(B,\mu)$ are well-posed, and that $X$ is a solution to the martingale problem for
$(A,\nu)$.  Then $Y=\gamma\circ X$ is a Markov process and a solution to the martingale problem for 
$(B,\mu)$.  Furthermore, for each $t\ge 0$ and $g\in B(F)$ we have, almost surely, 
$$E[g(X_t)|\ Y_s,\ 0\le s\le t]=\int_E g(x) \L (Y_t,dx).$$
\end{thm}
We remark that, under the hypotheses of the above theorem, $X$ is a Markov process and the forward equation
$$\nu_tf=\nu f+\int_0^t \nu_s Af ds,\qquad f\in \D(A)$$
has a unique continuous solution in $\P(D)$;  also the assumption of uniqueness for the martingale problem
for $(B,\mu)$ is not necessary, as it is implied by the other hypotheses; we refer the reader to~\cite{kurtz} for more details.

\end{document}